\documentclass[11pt,twoside]{amsart}

\usepackage{amsmath}
\usepackage{amsthm}
\usepackage{amsfonts, amssymb}
\usepackage{mathrsfs}
\usepackage[all,line]{xy}
\usepackage{url}
\usepackage{graphics}
\usepackage{epstopdf}
\usepackage{comment}
\usepackage{xcolor}

\definecolor{thmcol}{RGB}{65, 102, 245}
\definecolor{citecol}{RGB}{65, 102, 245}
\definecolor{linkcol}{RGB}{65, 102, 245}
\definecolor{urlcol}{RGB}{65, 102, 245}

\usepackage[colorlinks=true]{hyperref}
\hypersetup{
    colorlinks,
    citecolor=citecol,
    linkcolor=linkcol,
    urlcolor=urlcol
}

\usepackage{latexsym}
\usepackage[dvips]{graphicx}

\usepackage{tikz}
\usepackage{tkz-graph}
\usetikzlibrary{decorations.text,arrows,quotes}

\usepackage{tikz-cd}

\theoremstyle{plain}
\newtheorem{lemma}{Lemma}[section]
\newtheorem{proposition}[lemma]{Proposition}
\newtheorem{thm}[lemma]{Theorem}
\newtheorem*{thm1}{Theorem \hypersetup{linkcolor=black}\ref{principal}}
\newtheorem*{prop1}{Proposition \hypersetup{linkcolor=black}\ref{weak}}

\newtheorem*{cor1}{Theorem \hypersetup{linkcolor=black}\ref{coro}}
\newtheorem*{cor2}{Corollary \hypersetup{linkcolor=black} \ref{corolot}}
\newtheorem{corollary}[lemma]{Corollary}
\theoremstyle{remark}

\newtheorem{remark}[lemma]{Remark}

\theoremstyle{definition}
\newtheorem{definition}[lemma]{Definition}
\newtheorem{ej}[lemma]{Example}

\def\k{\mathcal{K}}

\def\Z{\mathbb{Z}}

\def\pe{\mathcal{P}}
\def\kp{\k_{\mathcal{P}}}

\begin{document}

\title[Weakly $G$-slim complexes  and the non-positive immersion property]{Weakly $G$-slim complexes and the non-positive immersion property for generalized Wirtinger presentations}
\author[A.N. Barreto and E.G. Minian]{Agustín Nicolás Barreto and El\'\i as Gabriel Minian}

\address{Departamento  de Matem\'atica-IMAS (CONICET)\\
 FCEyN, Universidad de Buenos Aires\\ Buenos
Aires, Argentina.}

\email{abarreto@dm.uba.ar}
\email{gminian@dm.uba.ar}

\thanks{Researchers of CONICET. Partially supported by grants PICT 2019-2338 and UBACYT 20020190100099BA}

\begin{abstract}
 We investigate weakly $G$-slim complexes, a more flexible variant of Helfer and Wise's slim complexes, which can be defined on any regular $G$-covering. We prove that if a $2$-complex  $X$ associated to a group presentation is weakly $H$-slim for some regular $H$-covering, and $\pi_1X$ is left-orderable, then $X$ is slim and, in particular, it has non-positive immersions. We apply this result to study conditions on generalized Wirtinger presentations that guarantee the non-positive immersion property for their associated $2$-complexes. This class of presentations includes classical Wirtinger presentations of (perhaps high-dimensional) knots in spheres, some classes of Adian presentations and LOTs presentations. On the one hand, our results provide a wide variety of examples of presentation complexes with the non-positive immersion property via a condition that can be checked with a simple algorithm. On the other hand, as an application of these methods, we derive an extension of a result by Wise on Adian presentations and prove a stronger formulation of a result of Howie on LOT presentations. 
\end{abstract}



\maketitle

\section{Introduction}
An immersion is a map between $2$-complexes that is locally injective. A $2$-complex $X$ has non-positive immersions if for any immersion $Y \looparrowright X$ with $Y$ compact and connected, either $\chi(Y) \leq 0$ or $Y$ is contractible. This notion is a topological variant of non-positive curvature and was introduced by Wise in \cite{wi0,wi} (see also \cite{hewi,wi2}). Note that in \cite{wi2}, this property is referred to as \textit{contracting} non-positive immersions. Wise showed that the fundamental group of a $2$-complex with non-positive immersions is locally indicable \cite{wi2}. Recall that a group $G$ is locally indicable if every finitely generated subgroup of $G$ admits a non-trivial homomorphism to the integers. Similarly as in the classical geometric context of non-positive curvature, the non-positive immersion property is related to asphericity: if $X$ is a $2$-complex with non-positive immersions and $\pi_1X$ is non-trivial, then $X$ is aspherical (see \cite[Corollary 2.8]{jali}). It is conjectured that if $X$ has non-positive immersions then $\pi_1X$ is coherent, which means that every finitely generated subgroup is finitely presented \cite[Conjecture 12.11]{wi1} (see also \cite{wi2}). In fact, this was part of Wise's original plan to prove Baumslag's conjecture on the coherence of one-relator groups \cite{wi0,wi1,wi2}). Helfer and Wise \cite{hewi}, and independently Louder and Wilton \cite{lw0}, proved that (the $2$-complexes associated to) torsion-free one-relator presentations have the non-positive immersion property. Wise \cite{wi2}, and independently Louder and Wilton \cite{lw}, showed that one-relator groups with torsion are coherent. In \cite{hs} Howie and Short extended the results of Wise, Louder and Wilton on coherence in one-relator groups with torsion to one-relator products of locally indicable groups, and exhibited an alternative proof of the non-positive immersion property for torsion-free one-relator groups. More recently Jaikin-Zapirain and Linton proved the coherence of all one-relator groups \cite{jali}, and showed that the fundamental group of any $2$-complex with non-positive immersions is homologically coherent. It is worth noting that it is unknown whether there exists a group which is homologically coherent but not coherent (see \cite{jali}).

 Other known examples of complexes with non-positive immersions are staggered $2$-complexes without proper powers \cite{hewi}, and complexes satisfying Sieradski's coloring test. In particular, the $2$-complexes associated to cycle-free Adian presentations have this property \cite{wi}. Also, any compact irreducible $3$-manifold with boundary has a two-dimensional spine with non-positive immersions \cite{wi2}. 
Although significant progress has been made, it remains challenging to algorithmically describe broad classes of group presentations whose standard 2-complexes have non-positive immersions. Wise observed that the proofs of the non-positive immersion property for specific 2-complexes are often ad hoc and asked whether there is an algorithm to detect this property \cite[Problem 1.9]{wi2} (see also \cite[Problem 12.2]{wi1}). In this direction, Wilton introduced the notion of maximal irreducible curvature $\rho_+(X)$, a new curvature invariant which is algorithmically computable. He showed that if $\rho_+(X)\leq 0$, then $X$ has non-positive immersions, and asked whether there exists an example of a $2$-complex $X$ with non-positive immersions such that $\rho_+(X)>0$. In \cite{blm} Blufstein and Minian provided such an example, which shows that these notions are not equivalent.

  In \cite{hewi} Helfer and Wise introduced the notion of a slim complex and proved that these $2$-complexes have non-positive immersions. A $2$-complex $X$ is slim if there is a $\pi_1$-invariant preorder on the set of $1$-cells of its universal cover $\tilde X$ satisfying certain properties on the edges of the boundaries of the $2$-cells of $\tilde X$ (see Definition \ref{slim} below).  In many situations it is more convenient to work with other regular coverings of $X$ different from the universal cover. This happens, for example, in the case of the complexes $\k_{\mathcal{P}}$ associated to generalized Wirtinger presentations $\pe$. A generalized Wirtinger presentation is a  group presentation $\pe= \langle a_1, \ldots, a_n \mid  r_1, \ldots, r_k \rangle$ with cyclically reduced relators such that $H_1(\k_{\mathcal{P}})$ is a nontrivial free abelian group of rank $n-k$. These presentations include classical Wirtinger presentations of (perhaps high-dimensional) knots in spheres, some classes of Adian presentations and LOTs presentations. In those cases it is more natural to define $\Z$-invariant preorders in infinite cyclic coverings, instead of working with the universal cover. In this paper we introduce the notion of weakly $G$-slim complex, extending the concept of slim complex to any regular covering. The definition of weakly $G$-slim complex also replaces one of the conditions of slim complexes by a more flexible condition that involves elementary reductions. We remark that the notion of reduction that we use is the one introduced by Howie in \cite{h1} (see also \cite{h2}), which is less restrictive than the one used by Helfer and Wise in \cite{hewi}, and is more suitable for the applications developed in Section \ref{applications}. The first result of this article asserts that if the $2$-complex  $X=\k_{\pe}$ associated to a presentation $\pe$ is weakly $H$-slim for some regular $H$-covering and $\pi_1X$ is left-orderable, then $X$ is slim. 

\begin{thm1} 
	Let $\mathcal{P}$ be a group presentation. Let $X = K_{\mathcal{P}}$ and let $\hat{X} \to X$ be an $H$-covering. Suppose  that $\pi_1X$ is left-orderable. If $X$ is weakly $H$-slim, then $X$ is slim. In particular, $X$ has the non-positive immersion property.
\end{thm1}

In our previous article \cite{bam1} we introduced an easy-to-check  condition on the relators, called weak concatenability (see Definition \ref{weakconca} below). We proved that, for generalized Wirtinger presentations, weak concatenability implies local indicability of the presented group \cite[Theorem 2.18]{bam1}. With a similar proof, one can show that it also implies weak $\Z$-slimness. 

\begin{prop1} 
Let $\mathcal{P} = \langle a_1, \ldots, a_n \mid  r_1, \ldots, r_k \rangle$ be a presentation of a group $G$, with cyclically reduced relators and such that $H_1(G)$ is a nontrivial free abelian group of rank $n-k$. Let $\varphi : G \to \mathbb{Z}$ be a surjective homomorphism. If $\{r_1, \ldots , r_k\}$ is weakly concatenable with respect to $\varphi$, then $\mathcal{K}_{\mathcal{P}}$ is weakly $\mathbb{Z}$-slim.
\end{prop1}

 As a consequence of Proposition \ref{weak} and Theorem \ref{principal}, we obtain the second result of the paper. 

\begin{cor1}
   Let $\mathcal{P} = \langle a_1, \ldots, a_n \mid  r_1, \ldots, r_k \rangle$ be a group presentation with cyclically reduced relators of a group $G$ with $H_1(G)$ nontrivial free abelian of rank $n-k$. Let $\varphi : G \to \mathbb{Z}$ be a surjective homomorphism. If $\{r_1, \ldots , r_k\}$ is weakly concatenable with respect to $\varphi$, then $\k_{\mathcal{P}}$ has the non-positive immersion property.
\end{cor1}

On the one hand, Theorem \ref{coro} provides a wide variety of examples of presentation complexes with non-positive immersions via a condition that can be checked with a simple algorithm. On the other hand, it can be used to investigate certain classes of Adian presentations, extending a result by Wise \cite[Section 11.2]{wi} (see Theorem \ref{adian} below). In particular we apply this result to study LOT presentations. These presentations are related to ribbon disk complements and Whitehead's Asphericity Question (see \cite{h15,h2} for more details). It is not known whether LOT groups are locally indicable (in fact, it is not known whether they are, in general, torsion-free). Local indicability of LOT groups would imply the asphericity of their associated complexes, since they have the homology of a circle (see \cite{h2}). Wise conjectured that the associated complexes of LOT presentations satisfy the non-positive immersion property (\cite[Conjecture 12.21]{wi1}). Our last result goes in that direction. To each LOT $\Gamma$ one can associate
two graphs: $I(\Gamma)$ and $T(\Gamma)$ (see Section \ref{secadian} below). Howie proved that, if $I(\Gamma)$ or $T(\Gamma)$ has no cycles, then the group $G(\Gamma)$ presented by $P(\Gamma)$ is locally indicable \cite{h2}. Applying Theorem \ref{coro}, we prove the following stronger result. 
\begin{cor2}
	Let $\Gamma$ be a reduced LOF. If $I(\Gamma)$ or $T(\Gamma)$ has no cycles, then the associated $2$-complex $\k_{\mathcal{P}_\Gamma}$ has the non-positive immersion property. 
\end{cor2}

\section{$G$-coverings and weakly $G$-slim complexes}

A $2$-complex is a two-dimensional CW-complex. We will work in the category of combinatorial $2$-complexes and combinatorial maps, which send (open) $n$-cells homeomorphically to $n$-cells. Let $G$ be a group. By a $G$-covering we mean a regular covering $\hat X\to X$ of a $2$-complex $X$ with covering (deck) transformation group isomorphic to $G$. Given a group presentation $\mathcal{P}$, let $X = K_{\mathcal{P}}$ be the associated standard $2$-complex, with one $0$-cell, one $1$-cell for each generator, and one $2$-cell for each relator of $\mathcal{P}$. A surjective homomorphism $\varphi:\pi_1X\to G$ determines a $G$-covering $\rho: \hat{X} \to X$, whose fundamental group is isomorphic to $\ker \varphi$ (via $\rho$). We index the $0$-cells of $\hat X$ with the elements of $G$ as follows. We choose a base point (a $0$-cell), which we denote by $1$. For every $g\in G$, we choose a loop $\gamma$ such that $\varphi([\gamma]) = g$, and we let $g$ be the endpoint of the lift $\hat \gamma$ of $\gamma$ that starts at the vertex $1$. Note that $g = g \cdot 1$, and, in general, $g \cdot h = gh$  with the usual action of $G$ on $\hat {X}$ by covering transformations. The $1$-cells of $\hat X$ are indexed similarly:  the lift of every generator $a$ of $\mathcal{P}$ (=$1$-cell of $X$) starting at the vertex $g$ is denoted by $\hat{a}_{g}$. Note that $g \cdot \hat{a}_h = \hat{a}_{gh}$. 
When $G$ is the group presented by $\pe$ (which we denote by $G(\mathcal{P})$), the $G$-covering of $X$ (the universal cover) is the Cayley complex associated to $\pe$, and its $1$-skeleton is the Cayley graph of $\pi_1X$ with respect to the generator set of $\mathcal{P}$.

It is not difficult to see that, if $\rho: \Tilde{X} \to X$ is the universal cover of $X=K_{\pe}$ and  $\rho_1: \hat{X} \to X$ is the $H$-covering associated to a surjective homomorphism $\varphi : \pi_1X \to H$, the unique (covering) map $\rho_2: \Tilde{X} \to \hat{X}$ such that  $\rho_1\rho_2=\rho$ and $\rho_2(1_{G(P)})=1_H$, satisfies $\rho_2(g) = \varphi(g)$ and $\rho_2(\Tilde{a}_g)=\hat a_{\varphi(g)}$ for each $g \in G(P)$.

The notion of slim $2$-complex was introduced by Helfer and Wise in \cite{hewi}. They proved that slim complexes have non-positive immersions \cite[Corollary 5.6]{hewi}. Slimness is defined via a preorder on the edges of the universal covering of $X$.

\begin{definition} \label{slim}
A combinatorial $2$-complex $X$ is called slim if its universal cover $\tilde X$ satisfies the following conditions.
\begin{enumerate}
    \item There is a $\pi_1 X$-invariant preorder on the set $E(\Tilde{X})$ of edges of $\tilde X$.
    \item For each $2$-cell $\Tilde{R}$ of $\Tilde{X}$,   the set of edges of its boundary $E(\partial \Tilde{R})$ has a unique strictly minimal $1$-cell $\min(\Tilde{R})$. Moreover, the attaching path of $\Tilde{R}$ traverses $\min(\Tilde{R})$ exactly once.
    \item If $\Tilde{R}_1 \neq \Tilde{R}_2$ are $2$-cells of $\Tilde{X}$ and $\min(\Tilde{R}_1)$ appears on $\partial \Tilde{R}_2$, then $\min(\Tilde{R}_1) > \min(\Tilde{R}_2)$.
\end{enumerate}
\end{definition}

We will use a somewhat weaker notion of slimness. This notion can be defined for any regular covering of $X$, and replaces condition $(2)$ of above by a weaker condition that involves the concept of (elementary) reduction introduced by Howie 
\cite{h1, h2, hs} (see also \cite{hewi, li}). The main result of this section asserts that if $X=K_{\pe}$ is weakly $H$-slim for some regular $H$-covering and $\pi_1X$ is left-orderable, then $X$ is slim (and therefore it has non-positive immersions). Recall that a group is left-orderable if it admits a total order that is invariant under left multiplication. We recall first some basic facts about reductions.

\begin{definition}
    A pair of combinatorial $2$-complexes $(X,Y)$ is said to be an elementary reduction if $X$ is obtained from $Y$ by attaching a $1$-cell $e$ and at most one $2$-cell $R$, and the attaching path of $R$ is not homotopic in $Y \cup \{e\}$ to a path in $Y$. We will also refer to the pair $(X,Y)$ as an $e$-elementary reduction (if only a $1$-cell $e$ is attached) or an $(R,e)$-elementary reduction (if also a $2$-cell $R$ is attached). We say that an $(R,e)$-elementary reduction is simple (or without proper powers) if the attaching map of $R$ is not a proper power in $\pi_1(Y\cup \{e\})$.
\end{definition}

A subcomplex $Y$ of a $2$-complex $X$ will be called a (simple) reduction of $X$ if $Y$ can be obtained from $X$ after countably many (simple) elementary reductions. A $2$-complex $X$ is reducible if every finite subcomplex $Y\subseteq X$ not contained in $X^{(0)}$ has an elementary reduction.

\begin{remark}
Note that the notion of reduction that we use is Howie's original one \cite{h1,h2}, which is more general than the one used by Helfer and Wise in \cite{hewi}. For instance, the infinite cyclic coverings associated to the generalized Wirtinger presentations of Proposition \ref{weak} and Theorem \ref{coro} are reducible in the sense of Howie, but not necessarily reducible in the sense of Helfer and Wise.
\end{remark}

\begin{definition} 
Let $X$ be a combinatorial $2$-complex and let $G$ be a group. We say that $X$ is weakly $G$-slim if there exists a $G$-covering $Z\to X$ which satisfies the following conditions.
\begin{enumerate}
    \item There is a $G$-invariant preorder on the set $E(Z)$ of edges of $Z$.
    \item For each $2$-cell $R$ of $Z$, the set $E(\partial R)$ has a unique strictly minimal $1$-cell $\min(R)$. Moreover, there exists a simple $(R,\min(R))$-elementary reduction $(W,T)$, where $W\subseteq Z$ is a simple reduction.
    \item If $R_1 \neq R_2$ are $2$-cells of $Z$ and $\min(R_1)$ appears in $\partial R_2$, then $\min(R_1) > \min(R_2)$.
\end{enumerate}
\end{definition}

We will need the following  two lemmas. The first one is due to Howie \cite[Corollary 3.4]{h1} (see also \cite[Lemma 7.2]{hewi}).

\begin{lemma} \label{essential}
	Let $(X,Y)$ be an $(R,e)$-elementary reduction with $\pi_1Y$ locally indicable. Suppose that $\partial_pR= P_1P_2$, with $P_i$ a closed path in $X$ that traverses $e$. Then each $P_i$ is essential in $X$.
\end{lemma}

The second one is an immediate consequence of \cite[Theorem 4.3]{h0} and \cite[Theorem 4.2]{h1} (see also \cite[Corollary 3.2]{li}). 

\begin{lemma}\label{injective}
	Let $Y$ be a simple reduction of $X$. Suppose that $\pi_1(Y,y)$ is locally indicable for each base point $y\in Y$. Then $Y \hookrightarrow X$ is $\pi_1$-injective on each component of $Y$.
\end{lemma}

Lemmas \ref{essential} and \ref{injective} are used to prove the following result about liftings of $1$-cells involved in elementary reductions. Given a $2$-cell $R$ we denote by $\partial_pR$ its attaching path and $\partial R$ the $1$-complex induced by the edges in $\partial_pR$.

\begin{proposition} \label{onetime}
	Let $(X,Y)$ be a (simple) $(R,e)$-elementary reduction and $\rho:\Tilde{X} \to X$ be the universal covering of $X$. If $\pi_1Y$ is locally indicable, the attaching path of any $2$-cell $\Tilde{R} \in \rho^{-1}(R)$ traverses each $1$-cell $\Tilde{e} \in \rho^{-1}(e)$ at most once.
\end{proposition}

\begin{proof}
	Suppose that the attaching path of a $2$-cell $\Tilde{R} \in \rho^{-1}(R)$ traverses an edge $\Tilde{e} \in \rho^{-1}(e)$ more than once. Then, we can either write the attaching path of $\Tilde{R}$ as $P_1\Tilde{e}P_2\Tilde{e}^{-1}$ where $P_1$ and $P_2$ are closed paths in $\Tilde{R}$, or we can write it as $P_1P_2$ where $P_1$ and $P_2$ are closed paths in $\Tilde{R}$ that traverse $\Tilde{e}$. In both cases we will use the fact that, since $\Tilde{X}$ is simply connected, $\rho_{*}([P_1])=\rho_{*}([P_2])=1_{\pi_1X}$. Now, for the latter case, note that $\partial_pR = \rho(P_1P_2) = \rho(P_1)\rho(P_2)$. This cannot occur since, by Lemma \ref{essential}, each $\rho(P_i)$ should be essential in $X$. In the first case, we have two possible situations. If neither $P_1$ nor $P_2$ traverses an edge in $\rho^{-1}(e)$, then $\partial_pR = \rho(P_1P_2) = \rho(P_1)e\rho(P_2)e^{-1}$. Since $\rho(P_1)$ is a closed path on $Y$ and $\rho_{*}([P_1]) = 1_{\pi_1X}$, by  Lemma \ref{injective} $\partial_pR$ is homotopic in $Y \cup \{e\}$ to $\rho(P_2)$, a path that does not traverse $e$. This contradicts the definition of elementary reduction. On the other hand, if $P_2$ traverse an edge $\Tilde{e}_2 \in \rho^{-1}(e)$, we can write $\partial_p\Tilde{R} = P_1\Tilde{e}P_2^{1}\Tilde{e}_2P_2^{2}\Tilde{e}^{-1}$ with $P_2^{1},P_2^{2}$ subpaths of $P_2$. Up to cyclic conjugation, we get that $\partial_pR = \rho(\Tilde{e}^{-1}P_1\Tilde{e})\rho(P_2^{1}\Tilde{e}_2P_2^{2})$. This is also a contradiction since, by Lemma \ref{essential}, the closed paths $\rho(\Tilde{e}^{-1}P_1\Tilde{e})$ and $\rho(P_2^{1}\Tilde{e}_2P_2^{2})$ should be essential in $X$, but $[\rho(\Tilde{e}^{-1}P_1\Tilde{e})]=\rho_{*}([\Tilde{e}^{-1}P_1\Tilde{e}])=1_{\pi_1X}$.
\end{proof}

\begin{corollary} \label{onetimecor}
	Let $X,Y,Z$ be combinatorial $2$-complexes where $Y$ is a connected simple reduction of $X$, and $(Y,Z)$ is a simple $(R,e)$-elementary reduction. Let $\rho:\Tilde{X} \to X$ be the universal covering of $X$. If $\pi_1Z$ is locally indicable, the attaching path of any $2$-cell $\Tilde{R} \in \rho^{-1}(R)$ traverses each $1$-cell $\Tilde{e} \in \rho^{-1}(e)$ at most once.
\end{corollary}

\begin{proof}
	Since $\pi_1Z$ is locally indicable and $(Y,Z)$ is a simple elementary reduction, by \cite[Theorem 4.2]{h1} $\pi_1Y$ is also locally indicable (see also \cite[Corollary 3.2]{li}).  By Lemma \ref{injective}, the inclusion $Y \hookrightarrow X$ is $\pi_1$-injective. If $\rho': \Tilde{Y} \to Y$ is the universal covering of $Y$, we can identify $\Tilde{Y}$ with a subcomplex $\Tilde{Y} \subset \Tilde{X}$. The result now follows from Proposition \ref{onetime} applied to the pair $(Y,Z)$.  
\end{proof}

\begin{proposition}\label{esreducible}
	Let $\mathcal{P}$ be a group presentation. Let $X = K_{\mathcal{P}}$ and let $\hat{X} \to X$ be an $H$-covering. If $X$ is weakly $H$-slim, then any subcomplex $Y\subseteq \hat X$ is reducible without proper powers. In particular, $\pi_1Y$ is locally indicable.	
\end{proposition}

\begin{proof}
	It suffices to show that any finite subcomplex $K\subset \hat X$ of dimension $2$ has a simple elementary reduction (see \cite[Section 2]{h1}). Given such a subcomplex $K$, take a $2$-cell $R$ of $K$ such that $\min(R)$ is minimal in the set of $1$-cells of $K$. By definition, there is a simple $(R,\min(R))$-elementary reduction $(W,T)$ where $W$ is a simple reduction of $\hat X$. If $K\subseteq W$, we are done. If not, since $W\subset \hat X$ is a simple reduction, there exists a minimal simple reduction $W\subset Y$ containing $K$. Then $Y$ has a simple elementary reduction which is also a  simple elementary reduction in $K$.   
\end{proof}

\begin{thm} \label{principal}
	Let $\mathcal{P}$ be a group presentation. Let $X = K_{\mathcal{P}}$ and let $\hat{X} \to X$ be an $H$-covering. Suppose  that $\pi_1X$ is left-orderable. If $X$ is weakly $H$-slim, then $X$ is slim. In particular, $X$ has the non-positive immersion property.
\end{thm}

\begin{proof}
	Let $G=\pi_1X$ and let $\varphi:\pi_1X\to H$ be a surjective homomorphism inducing the $H$-covering  $\hat{X} \to X$. We will use the notation introduced at the beginning of this section. Let $\rho: \Tilde{X} \to X$ be the universal cover. The (covering) map $\rho_2:\Tilde{X}\to \hat{X}$ satisfies $\rho_2(g) = \varphi(g)$ and $\rho_2(\Tilde{a}_g)=\hat a_{\varphi(g)}$ for each $g$ and $\Tilde{a}_g$.

	We define first a preorder on the $1$-cells of $\Tilde{X}$ using the preorder on the $1$-cells of $\hat{X}$ and the order of $G$ as follows. Note that every $1$-cell of $\Tilde{X}$ has the form $\Tilde{a}_{g}$ for some generator $a$ of $\mathcal{P}$ and some $g \in G$. We assign to each  $\Tilde{a}_{g}$ the pair $(\hat a_{\varphi(g)},g)$. This assignment is injective. We equip the $1$-cells of $\tilde X$ with the lexicographic order with respect to their associated pairs. Note that this preorder is $\pi_1X$-invariant. This follows from the $H$-invariance of the preorder of the $1$-cells of $\hat{X}$ (and the fact that $\rho_2(g' \cdot \Tilde{a}_{g}) = \varphi(g')\cdot \hat a_{\varphi(g)}$) and the left-orderability of $G$
	
	We prove now that there is a unique strictly minimal $1$-cell in $E(\partial \tilde R)$ for any $2$-cell $\tilde R$ of $\Tilde{X}$. Given $\tilde R$, let  $\Tilde{a}_{g}$  be a minimal $1$-cell in $E(\partial \tilde R)$. We have to see that it is unique. Since $\rho_2(\partial \tilde R) = \partial \rho_2(\tilde R)$ and  $\Tilde{a}_{g}$ is minimal, the first coordinate $\hat a_{\varphi(g)}$ of the pair assigned to  $\Tilde{a}_{g}$ must be $\min(\rho_2(\tilde R))$. This means that the candidates for minimal $1$-cells in $E(\partial \tilde R)$ are in the same orbit $\rho_2^{-1}(\min(\rho_2(\tilde R)))$, and the left-orderability of $\pi_1X$ induces a total order for these $1$-cells. Therefore, there is a unique strictly minimal $1$-cell in $E(\partial \tilde R)$ and $\min(\tilde R)$ is well-defined. 
	The second part of condition $(2)$ in the definition of slimness follows from Proposition \ref{esreducible} and Corollary \ref{onetimecor}. 
	
	It remains to prove that condition $(3)$ holds for any pair of $2$-cells $\Tilde{R}_1$ and $\Tilde{R}_2$. Suppose that $\min(\Tilde{R}_1) = \min(\Tilde{R}_2)$. If $\rho_2(\Tilde{R}_1) \neq \rho_2(\Tilde{R}_2)$ we get immediately a contradiction since $\min(\rho_2(\Tilde{R}_1)) = \rho_2(\min(\Tilde{R}_1))$ and $\min(\rho_2(\Tilde{R}_2)) = \rho_2(\min(\Tilde{R}_2))$ but $\min(\rho_2(\Tilde{R}_1))\neq \min(\rho_2(\Tilde{R}_2))$. If $\rho_2(\Tilde{R}_1) = \rho_2(\Tilde{R}_2)$, then $\Tilde{R}_1 = g\cdot \Tilde{R}_2$ for some $g \in G$ with $g \neq 1$. The $\pi_1X$-invariance of the given preorder implies that $\min(\Tilde{R}_1) = g\cdot \min(\Tilde{R}_2) \neq \min(\Tilde{R}_2)$, which is again a contradiction.
\end{proof}

\begin{remark}\label{skipleftord}
Note that, in the statement of Theorem \ref{principal}, the condition on the left- orderability of $\pi_1 X$ can be dropped if we require the group $H$ to be locally indicable. In that case, $\pi_1X$ is automatically locally indicable (and, in particular, left-orderable). This follows from the short exact sequence $1\to \ker \varphi \to \pi_1X\to H\to 1$ and Proposition \ref{esreducible}.
\end{remark}

\section{The non-positive immersion property for generalized Wirtinger presentations}\label{applications}

Let $\mathcal{P}$ be a (finite) presentation of a group $G$ and $\varphi : G \to \mathbb{Z}$ a surjective homomorphism. By changing, if necessary, a generator by its inverse, we can assume that $\varphi(a)$ is non-negative for every generator $a$ of $\mathcal{P}$. Given a word $w$ on the generators,  $\varphi(w)$ is called the weight of $w$ (with respect to $\varphi$). For each relator $r$, we consider the sequence of weights of initial subwords of $r$ and define the multiset $m_{\varphi}(r)$ of its minima as follows. If the minimum of the sequence is $m$, then $m_{\varphi}(r)$ contains a copy of the letter (generator) $a$ for every initial subword $wa^{-1}$ with weight $m$, and we call them $\textit{negative copies}$, and another copy of $a$ for every initial subword $wa$ with weight $m + \varphi(a)$, and we call them $\textit{positive copies}$. We illustrate this with an easy example.

\begin{ej} \label{pres}
	Let $\mathcal{P} = \langle a, b, c \mid a^{-1}b, \ c^{-1}b^{-1}caba^{-1}c^{-1}b^{-1}c^2 \rangle$ and define $\varphi : G(P) \to \mathbb{Z}$ by $\varphi(a) = \varphi(b) = \varphi(c) = 1$. The sequence for $r_2 = c^{-1}b^{-1}caba^{-1}c^{-1}b^{-1}c^2$ is $$-1, -2, -1, 0, 1, 0, -1, -2, -1, 0$$ and $m_{\varphi}(r_2) = \{b,c,b,c\}$ with two negative copies of $b$ and two positive copies of $c$. The sequence for $r_1 = a^{-1}b$ is $-1, 0$ and $m_{\varphi}(r_1) = \{a,b\}$ with a negative copy of $a$ and a positive copy of $b$.
\end{ej}

\begin{definition}\label{weakconca}
	A family of relators $\{r_1,\ldots, r_k\}$ is weakly concatenable (with respect to $\varphi : G \to \mathbb{Z}$), if there is an ordering $m_{\varphi}(r_{i_1}), \ldots , m_{\varphi}(r_{i_k})$ of their multisets of minima such that, for every $1 \leq j \leq k$, there exists an element $x \in m_{\varphi}(r_{i_j})$ such that $x \notin \cup_{s=1}^{j-1} m_{\varphi}(r_{i_s})$ and the number of positive and negative copies of $x$ in $m_{\varphi}(r_{i_j})$ are different.
\end{definition}

 The condition on the number of positive and negative copies of $x$ is used in the next proposition to prove that a certain lift $R_{j,i}$ of the $2$-cell $R_i$ (corresponding to
a relator $r_i$) properly involves a lift of the $1$-cell corresponding to the generator $x$.

\begin{proposition} \label{weak}
Let $\mathcal{P} = \langle a_1, \ldots, a_n \mid  r_1, \ldots, r_k \rangle$ be a presentation of a group $G$, with cyclically reduced relators and such that $H_1(G)$ is a nontrivial free abelian group of rank $n-k$. Let $\varphi : G \to \mathbb{Z}$ be a surjective homomorphism. If $\{r_1, \ldots , r_k\}$ is weakly concatenable with respect to $\varphi$, then $\mathcal{K}_{\mathcal{P}}$ is weakly $\mathbb{Z}$-slim.
\end{proposition}

\begin{proof}
Let $X = \mathcal{K}_{\mathcal{P}}$ and let $\hat X\to X$ be the cyclic covering of $X$ determined by $\varphi$. We denote  by $a_{j,i}$ the lift of the edge $a_i$ beginning at the vertex $j$. Let $R_i$ be the $2$-cell of $X$ corresponding to the relator $r_i$ and denote by $R_{j,i}$ the $2$-cell of $\hat{X}$ that covers $R_i$ for which the minimum $0$-cell is $j$.

We define a total ordering on the edges  $a_{j,i}$ of $\hat{X}$ as follows. First, for simplicity, we reindex the $k$  relators of $\pe$ (and the $2$-cells of $ \mathcal{K}_{\mathcal{P}}$) as $ r_{n-k+1}, \ldots, r_n$ (and $R_{n-k+1}, \ldots, R_n$). By reordering the relators of $\pe$ if necessary, we can assume that the weak concatenability occurs in such a way that $\min(R_{j,i})=a_{j,i}$. The ordering on the edges  $a_{j,i}$ is the lexicographic order on the pairs $(j,i)$. The first coordinate with the usual order of $\Z$, and the second one with the inverse of the usual order on $\{1,2,\ldots,n\}$. It is immediate to check that it is a $\mathbb{Z}$-invariant ordering.  Given $(j,i)$,  we consider the full subcomplex of $\hat X$ containing all the $0$-cells $t \geq j$, and remove from it the $2$-cells $R_{j,s}$ for $s\geq i+1$ and the $1$-cells $a_{j,s}$ for $s\geq i+1$. Note that we obtain a well defined subcomplex, which we denote by $\hat{X}_{j,i}$, since the  $1$-cells $a_{j,s}$ removed do not lie in the boundaries of the remaining $2$-cells $R_{j,l}$. Note also that  $(\hat{X}_{j,i},\hat{X}_{j,i-1})$  is an $(R_{j,i},a_{j,i})$-elementary reduction if $i\geq n-k+1$, since $\hat{X}_{j,i}$ is obtained from $\hat{X}_{j,i-1}$ by attaching the $2$-cell $R_{j,i}$ and the $1$-cell $a_{i}$, which is properly involved by definition of weak concatenability. For $1\leq i\leq n-k$, $(\hat{X}_{j,i},\hat{X}_{j,i-1})$ is just an $(a_{j,i})$-elementary reduction. All the reductions involved are without proper powers since the presentation complex has the homology of a wedge of circles and $\Z$ is conservative \cite [Section 5]{h2} (see also \cite [Theorem 2.5]{bam1}). If $R_1 \neq R_2$ are $2$-cells of $\hat{X}$, then $R_1 = R_{j_1,i_1}$ and $R_2 = R_{j_2,i_2}$ for some pairs of integers $(j_1,i_1) \neq (j_2,i_2)$. Clearly $\min(R_1)=a_{j_1,i_1} \neq a_{j_2,i_2} = \min(R_2)$.
\end{proof}

Combining Proposition \ref{weak} with Theorem \ref{principal}, we derive the main result of this section.

\begin{thm}\label{coro}
    Let $\mathcal{P} = \langle a_1, \ldots, a_n \mid  r_1, \ldots, r_k \rangle$ be a presentation of a group $G$, with cyclically reduced relators and with $H_1(G)$ nontrivial free abelian of rank $n-k$. Let $\varphi : G \to \mathbb{Z}$ be a surjective homomorphism. If $\{r_1, \ldots , r_k\}$ is weakly concatenable with respect to $\varphi$, then $ \mathcal{K}_{\mathcal{P}}$ has the non-positive immersion property.
\end{thm}

\begin{proof}
	This follows from Proposition \ref{weak}, Theorem \ref{principal}, and Remark \ref{skipleftord}.
\end{proof}

 The presentation of Example \ref{pres} has the homology of a circle and $\{r_1,r_2\}$ is weakly concatenable with order $m_{\varphi}(r_1), m_{\varphi}(r_2)$. By Theorem \ref{coro}, the standard $2$-complex associated to $\mathcal{P}$ has non-positive immersions.

\begin{ej}
Consider the presentation $\mathcal{P} = \langle a, b, c \mid r_1, r_2 \rangle$ with relators  
\[
r_1 = a^{-1}c^{-1}aaab^{-1}b^{-1}c^{-1}ab, \quad r_2 = c^{-1}b^{-1}cb^{-1}caba^{-1}ba^{-1}.
\]
Let $\varphi : G \to \mathbb{Z}$ be the homomorphism defined by $\varphi(a) = \varphi(b) = \varphi(c) = 1$. It is easy to verify that the multisets of minima are $m_\varphi(r_1) = \{c, a, c, a\}$ and $m_\varphi(r_2) = \{b, c, b, c\}$, and that $\{r_1, r_2\}$ is weakly concatenable with the ordering $m_\varphi(r_2), m_\varphi(r_1)$. By Theorem \ref{coro}, $\mathcal{P}$ has the non-positive immersion property.
\end{ej}

Proposition \ref{weak} and Theorem \ref{coro} can be generalized by replacing the map $G\to \Z$ with a map $\varphi:G\to H$ onto any locally indicable group $H$. In that case, we replace the infinite cyclic cover by the corresponding $H$-covering. Note that the notion of weak concatenability can be naturally extended to any left-orderable group (and, in particular, to any locally indicable group). The no proper power condition used in the proof of Proposition \ref{weak} is still satisfied, since locally indicable groups are conservative (see \cite{ge0, hsc}). As a consequence, we obtain a more general version of Theorem \ref{coro}.

\begin{thm} \label{corogeneral}
	Let $\mathcal{P} = \langle a_1, \ldots, a_n \mid  r_1, \ldots, r_k \rangle$ be a presentation of a group $G$ with cyclically reduced relators and $H_1(G)$ nontrivial free abelian of rank $n-k$. Let $\varphi : G \to H$ be a surjective homomorphism onto a locally indicable group. If $\{r_1, \ldots , r_k\}$ is weakly concatenable with respect to $\varphi$, then $\mathcal{K}_{\mathcal{P}}$ has the non-positive immersion property.
\end{thm}

The following example shows a presentation $\mathcal{P}$ for which the weak concatenability condition is not satisfied for any surjective homomorphism $\varphi: G(\mathcal{P}) \to \mathbb{Z}$, but Theorem \ref{corogeneral} works with other locally indicable group $H$.

\begin{ej}
	Let $\mathcal{P} = \langle x, y, z  \mid  x^{-1}z^4xz^{-3}yzy^{-1}z^{-1}y^{-1} , y^{-1}x^{-1}y^{-1}z^{-1}xzyzxz^{-1}\rangle$ be a presentation of a group $G$ and let $\langle x, y, z  \mid  xzx^{-1}z^{-1}, xyxy^{-1}x^{-1}y^{-1}, yzyz^{-1}y^{-1}z^{-1}\rangle$ be the standard presentation of the Braid group $B_4$ on four strings. Consider the opposite of the Dehornoy ordering on $B_4$ \cite{cr}. Let $\varphi:G\to B_4$ be the map which sends each generator of $\pe$ to the generator of $B_4$ named with the same letter. The multisets of minima are $m_{\varphi}(r_1) = \{y,z\}$ and $m_{\varphi}(r_2) = \{x,z\}$. Note that $\{r_1,r_2\}$ is weakly concatenable. Since $H_1(G) = \mathbb{Z}$ and $B_4$ is locally indicable, $\k_{\pe}$ has the non-positive immersion property. However, there is no surjective map $\psi : G \to \mathbb{Z}$ making $\{r_1,r_2\}$ weakly concatenable. Indeed, any map $\psi : G \to \mathbb{Z}$ satisfies $\psi(x)=\psi(y)=\psi(z)$, and, if it is surjective, it must  send every generator of $\mathcal{P}$ to $1$ or to $-1$. An easy computation shows that, in both cases, $m_{\psi}(r_1) = \{x,z\}$ and $m_{\psi}(r_2) = \{x,z\}.$
\end{ej}

\section{Adian presentations and LOGs}\label{secadian}

We apply the results of the previous section to derive an extension of a result of Wise on Adian presentations \cite[Section 11.2]{wi} and a stronger formulation of  a well known result of Howie  on labelled oriented trees \cite[Theorem 10.1]{h2}.

 An Adian presentation is a group presentation of the form $$\pe=\langle a_1,\ldots,a_n \ | u_1=v_1,u_2=v_2,\ldots,u_k=v_k\rangle$$ where $u_i$ and $v_i$ are nontrivial positive words on $A=\{a_1,\ldots,a_n\}$. To each Adian presentation $\pe$ one can associate two graphs: $T(\pe)$ and $I(\pe)$. The vertex set of both graphs is the set of generators $A$. The edges of the graph  $T(\pe)$ (resp. $I(\pe)$) connect the first (resp. last) letter of $u_i$ to the first (resp. last) letter of $v_i$. Gersten proved that if the presentation is cycle-free (i.e. if both graphs are forests) then $\kp$ is diagrammatically reducible and, in particular, aspherical \cite[Proposition 4.12]{ge}. Moreover, when $\ell(u_i)=\ell(v_i)$ for every $i$, if either  $T(\pe)$ or $I(\pe)$ has no cycles then $\kp$ is diagrammatically reducible (and hence aspherical) \cite[Proposition 4.15]{ge}. More recently, Wise showed that the $2$-complexes associated to cycle-free Adian presentations have non-positive sectional curvature, in particular they have the non-positive immersion property \cite[Theorem 11.4]{wi}.  As an application of Theorem \ref{coro}, we obtain the following extension of Wise's result \cite[Theorem 11.4]{wi} for Adian presentations with $\ell(u_i)=\ell(v_i)$ for every $i$.

\begin{thm}\label{adian}
Let $\pe=\langle a_1,\ldots,a_n\ |\ u_1=v_1,\ldots,u_k=v_k\rangle$ be an Adian presentation with $H_1(G(\pe))$ nontrivial free abelian of rank $n-k$, and such that $\ell(u_i)=\ell(v_i)$ for every $i$. If either $I(\pe)$ or $T(\pe)$ has no cycles, then $\k_{\mathcal{P}}$ has the non-positive immersion property.
\end{thm}

\begin{proof}
Under the hypotheses on the Adian presentation, the map $\varphi:G(\pe)\to \Z$ which sends each generator to $1$ is a well defined surjective homomorphism. If $I(\pe)$ has no cycles, the weak concatenability condition is satisfied, and the result follows from Theorem \ref{coro}. If $T(\pe)$ has no cycles, the weak concatenability condition is satisfied by replacing minima with maxima.
\end{proof}

\begin{ej} Standard Artin presentations are particular examples of Adian presentations where the condition $\ell(u_i)=\ell(v_i)$ is satisfied. Recall that a presentation graph is  a finite simple graph $\Gamma$ with vertex set $S$ and the edges connecting vertices $s$ and $t$ are labelled by integers $m_{st}\geq 2$. If there are no edges between them, we set $m_{st} = \infty$. The Artin group defined by $\Gamma$ is the group $A_\Gamma$ with standard presentation:  
\[
P_\Gamma = \langle S \mid 
\underbrace{sts\cdots}_{m_{st} \text{ letters}} = \underbrace{tst\cdots}_{m_{st} \text{ letters}} \, \forall s,t \text{ such that } m_{st} \neq \infty \rangle.
\]

It is known that, when the defining graph $\Gamma$ is a forest, $A_\Gamma$ is a $3$-manifold group and, in particular, it is coherent \cite{go}. In almost all the other cases, $A_\Gamma$ is not coherent (see \cite[Proposition 9.24]{wi1}). In view of conjecture \cite[Conjecture 12.11]{wi1}, we cannot expect the non-positive immersion property of their associated complexes. If $\Gamma$ is a forest and all the labels  $m_{st}$ are odd, one can deduce the non-positive immersion property as follows.
Observe that, in that case, $H_1(A_\Gamma)$ is free abelian of rank $n-k$, where $n$ is the number of vertices and  $k$ is the number of edges of $\Gamma$. Moreover, the graph $I(\pe)$ coincides with $\Gamma$, and the non-immersion property follows from  Theorem \ref{adian}. Note, however, that this can be deduced also from \cite[Theorem 11.4]{wi} since, for Artin presentations, both graphs are equal and, in this case, $I(\pe)$ and $T(\pe)$ have no cycles.
\end{ej}

 Our methods are more effective when these graphs are different (and one of them has no cycles). This is the case of LOT presentations, which we study next.

\subsection*{Applications to LOTs}

LOG presentations are other particular cases of Adian presentations satisfying  $\ell(u_i)=\ell(v_i)$ for each $i$. A labelled oriented graph (LOG) $\Gamma$ consists of two sets $V(\Gamma)$ and $E(\Gamma)$ of vertices and edges, and three functions $i, t, \lambda : E \to V$ that map each edge to its initial vertex, terminal vertex and label respectively. A LOG is called a LOT (labelled oriented tree), or more generally a LOF (labelled oriented forest) if the underlying graph is a tree (resp. a forest). The standard presentation associated to a LOG $\Gamma$ is the following:
$$P(\Gamma) = \langle V(\Gamma)  \mid  \{t(e)^{-1}\lambda(e)^{-1}i(e)\lambda(e) : e \in E(\Gamma)\}\rangle.$$
Note that $P(\Gamma)$ has deficiency $|V(\Gamma)|-|E(\Gamma)|$ and, when the underlying graph is a forest, the presentation complex has the homology of a wedge of circles. 
We can assume that the LOFs are reduced (see \cite[Section 3]{h2}), which implies that all the relators of $P(\Gamma)$ are cyclically reduced (and then, $\ell(u_i)=\ell(v_i)=2$ for each $i$). The graphs $I(\pe)$ and $T(\pe)$ of Theorem \ref{adian} are denoted respectively $I(\Gamma)$ and $T(\Gamma)$, and they were investigated by Howie in \cite{h2}. Howie proved that if $I(\Gamma)$ or $T(\Gamma)$ has no cycles then the group $G(\Gamma)$ presented by $P(\Gamma)$ is locally indicable \cite{h2}. From Theorem \ref{adian} we deduce the following stronger result. 
\begin{corollary}\label{corolot}
    Let $\Gamma$ be a reduced LOF. If $I(\Gamma)$ or $T(\Gamma)$ has no cycles, then the associated $2$-complex $\k_{\mathcal{P}_\Gamma}$ has the non-positive immersion property. 
\end{corollary}

\end{document}